\documentclass{article}

\usepackage{bbm}
\usepackage[utf8]{inputenc}
\usepackage[english]{babel}
\usepackage{color}
\usepackage{amsmath,amssymb,amsthm,yhmath,amsfonts,mathrsfs,enumitem,mathtools,dsfont}
\usepackage{tikz}
\usetikzlibrary{shapes,arrows,positioning}

\newcommand{\pr}{\mathbb{P}}
\newcommand{\E}{\mathbb{E}}
\renewcommand{\d}{\mathrm{d}}

\newcommand{\1}[1]{{\mathds{1}}_{\left\{#1\right\}}}

\newcommand{\A}{\mathcal{A}}
\renewcommand{\S}{\mathcal{S}}

\newtheorem{theorem}{Theorem}
\newtheorem{definition}{Definition}

\newtheorem{lemma}{Lemma}
\newtheorem{prop}{Proposition}

\DeclareMathOperator{\cov}{cov}
\DeclareMathOperator{\diag}{diag}

\title{Markov processes on quasi-random graphs }

\author{
	D\'aniel Keliger\\
	{\small Department of Stochastics,}\\
	{\small Budapest University of Technology and Economics}\\
	{\small e-mail: perfectumfluidum@gmail.com}\\[3mm]
}

\date{}

\begin{document}
	
	
	\maketitle
	
	\begin{abstract}
		We study Markov population processes on large graphs, with the local state transition rates of a single vertex being linear function of its neighborhood. A simple way to approximate such processes is by a system of ODEs called the homogeneous mean-field approximation (HMFA). Our main result is showing that HMFA is guaranteed to be the large graph limit of the stochastic dynamics on a finite time horizon if and only if the graph-sequence is quasi-random. Explicit error bound is given and being of order $\frac{1}{\sqrt{N}}$ plus the largest discrepancy of the graph. For Erdős Rényi and random regular graphs we show an error bound of order the inverse square root of the average degree. In general, diverging average degrees is shown to be a necessary condition for the HMFA to be accurate. Under special conditions, some of these results also apply to more detailed type of approximations like the inhomogenous mean field approximation (IHMFA). We pay special attention to epidemic applications such as the SIS process.
	\end{abstract}

\section{Introduction}

Markov processes are valuable tools for modeling populations from  the individual level. Applications range from physics \cite{Glauber} , chemistry \cite{kurtz2011}, engineering \cite{Illes2017}, biology \cite{MarkovBiology} and social phenomena \cite{Short2008}. 

Due to the inherent dependencies between individuals exact analysis of such models are infeasible when the population is large. To mitigate the problem, one may focus on population level quantities and derive a deterministic ODE approximation for them by assuming the law of mass action holds. We call such technique the homogeneous mean-field approximation (HMFA). Based on the pioneering works of Kurtz \cite{kurtz70, kurtz78} HMFA describes the dynamics of macroscopic quantities accurately in the limit when the population size approaches infinite when said population is well mixed - every individual interacts with each other whit equal chance.  

Besides well mixed populations, attention to processes whit an underlying topology represented by a network connecting individuals who can interact whit each other also emerged. From the '70s processes like the voter model and the contact process were studied on lattices \cite{Contactprocess}. Later in the 2000s inspired by new developments in network theory extended these investigations to scale free networks such as Barabási-Albert graphs \cite{Barabasigraph, Bollobas2003}. For these networks, HMFA turns out to be a too rough approximation \cite{NagyNoemi2013}, hence, new methods have been developed retaining more details about the topology, yet keeping the resulting ODE system tractable for analytical and computational studies.

 An intermediate approximation was given by Pastor-Satorras and Vespignany for the Susceptible-Infected-Susceptibel (SIS) process  based on the heuristic that vertices with higher degrees are more prone to get infected \cite{vesp}. To account for the this heterogeneity, the authors grouped together vertices whit the same degree into classes and treated elements statistically indistinguishable resulting in an ODE system whit the number of equations proportional to the number of different degrees. We will refer to this method as the inhomogenous mean field approximation (IMFA).  

A more detailed approach called quenched mean field approximation or N-intertwined mean field approximation (NIMFA) takes into account the whole network topology via the adjacency matrix only neglecting the dynamical correlations between vertices. This results in a larger ODE system whit size proportional to the number of vertices this time.

Both IMFA and NIMFA were studied in depth yielding valuable insight about epidemic processes most notable about the epidemic threshold. To justify said approximations several numerical studies had been carried out comparing theoretical forecasts to stochastic simulations showing moderate deviations \cite{MeanfieldSimulationAccuracy, MeanfieldSimulationAccuracy2}.

Rigorous theoretical results showing convergence or error bounds, on the other hand, are few and far between. Examples are the above mentioned results for complete graphs by Kurtz \cite{kurtz70, kurtz78}, Volz's equation \cite{volz} for the Susceptible-Infected-Revered process (SIR) on graphs generated by the configuration model \cite{configuration} being the correct mean field limit \cite{volzproof, Volz_biz_hard}  and some exact bounding dynamics \cite{simon2017NIMFA, NicholasJ2018, Chen2020}. 

Our work was heavily inspired by \cite{UnifiedMeanfield} where concepts of discrepancy and spectral gap were utilized to bound one source of the error arising form mean field approximations called the topological approximation.

The aim of this paper is to carry out a rigorous analysis of the HMFA. The motivation behind dwelling into the accuracy of HMFA is two folds. 

Firstly, applications may includes graphs like expanders, where well mixing is only an approximation yet one expects them to perform well, therefore, explicit error bounds might be usefull for these settings.  

 Secondly, we think of HMFA as a stepping stone for understanding more detailed approximations whit the right balance of complexity, rich enough to be interesting and relevant but not too difficult to be unreachable. In some special cases more advanced approximation techniques can be reduces to HMFA. For example, when the graph is regular IMFA and HMFA gives the same approximation as there are only one degree class.

Our main contribution is to characterize the type of graph sequences for which the ODE given by HMFA is the correct large graph limit. These type of graph sequences are those for which the appropriately scaled discrepancy goes to $0$, called quasi random graphs sequences. Thus, we reduce the problem to a purely graph-theoretical one.

Explicit error bound is given containing the number of vertices and the largest discrepancy. For the later we provide an upper bound based on the spectral gap.

For two types of random graphs: Erdős Rényi graphs and random regular graphs, we show that the error of HMFA can be upper-bounded by a constant times the inverse square root of the average degree, making them accurate for diverging average degrees. We also show that in general, diverging average degree is a necessary condition for the HMFA to be accurate.

The paper is structured as follows: In section \ref{s:graph} we introduces concepts and notations relating graph theory including discrepancies and spectral properties. In section \ref{c:markov} we introduce the type of Markov processes which are evolving on the networks. These models are such that only neighboring vertices can influence directly each others transition rates in a linear fashion. This framework includes among many other models the SIS and the SIR process. In section \ref{s:HMFA} we introduce HMFA in detail and define precisely what we mean by the HMFA being accurate for a given graph sequence. Section \ref{s:results} states the theorems and propositions which for which the proof are given in Section \ref{s:proof}.

\newpage
\section{Graph properties}
\label{s:graph}

\subsection{Basic notations}

$\subset$ denotes subsets whit the convention of including proper subset.

$\left(G^n \right)_{n=1}^{\infty}$ is a sequence of graphs with $2 \leq N=N(n)$ vertices. The vertex set is labeled with $[N]=\{1, \dots, N\}$. We assume $$\lim_{n\to \infty}N(n)=\infty.$$

The adjacency matrix is denoted by $\A^n=\left(a_{ij}^n \right)_{i,j=1}^N$. $\A^n$ is symmetric and the diagonal values are $0$. The degree of vertex $i$ is denoted by
$$d^n(i)= \sum_{j=1}^N a_{ij}^n.$$
The average degree is denoted by
$$ \bar{d}^n:= \frac{1}{N}\sum_{j=1}^N d^n(i). $$
We assume $\bar{d}^n>0$.

For regular graphs, $d^n(i)=\bar{d}^n=d^n$.

The largest connected component (or one of the largest components, if there are several) is denoted by $V_{\textrm{conn}}^n$.
$$\theta^n:= 1-\frac{\left|V_\textrm{conn}^n \right |}{N}$$
denotes the ratio of vertices not covered by the largest connected component.

A subset $A \subset [N]$ is called an independent set if none of the vertices in $A$ are connected to each other. The size of the largest independent set is denoted by $\alpha^n$.

Based on the work of Caro and Wei \cite{CaroWei}, $\alpha^n$ can be estimated from below by
\begin{align*}
\alpha^n \geq \sum_{i=1}^N \frac{1}{d^n(i)+1}.
\end{align*}
Applying Jensen's inequality yields Turán's bound:
\begin{align}
\label{alpha_bound}
& \alpha^n \geq \frac{N}{\bar{d}^n+1}.
\end{align}

The number of edges between $A,B \subset [N]$ (counting those in $A \cap B$ double) is denoted by
$$e(A,B)=\sum_{i \in A}\sum_{j \in B}a_{ij}^n. $$
$e \left( \cdot, \cdot    \right )$ is symmetric and additive, that is,
$$e(A \sqcup B, C)=e(A,C)+e(B,C) $$
where $\sqcup$ denotes disjoint union. $e \left(\cdot, \cdot \right)$ is also increasing in both variables, meaning, for $A, B,C \subset [N], \ A \subset B$,
\begin{align*}
& e(A,C) \leq e(B,C).
\end{align*}
Therefore
\begin{align*}
& 0 \leq e(A,B) \leq e([N],[N])=N \bar{d}^n.
\end{align*}

Also, for any $A \subset [N]$,
$$e(A,[N])=\sum_{i \in A}d^n(i).$$

$e(A,A)=0$ for $A \subset [N]$ independent.

\subsection{Discrepancies}


In this section, we introduce several different measures of how well-mixed a graph is \cite{SparseQuasirandom}. This will be measured mostly by edge density in and between various subsets.


The \emph{discrepancy} between two subsets of vertices $A,B\subset [N]$ is
\begin{align}
\label{eq:delta}
& \delta(A,B):=\frac{e(A,B)}{N \bar{d}^n}-\frac{|A|}{N} \frac{|B|}{N}.
\end{align}

When $A, B$ are random sets whit given sizes $\E \left( e(A,B) \right) \approx  \frac{\bar{d}^n}{N} |A| \cdot |B|$ motivating the definition of \eqref{eq:delta}.

$\delta(A,B)$ inherits the symmetric and additive properties of $e(A,B)$, but not monotonicity, since $\delta(A,B)$ might take negative values.

From $0 \leq \frac{e(A,B)}{N \bar{d}^n},\frac{|A|}{N} \frac{|B|}{N} \leq 1 $ it follows that
\begin{align*}
& 0 \leq | \delta(A,B)| \leq 1.
\end{align*}

Our main focus is the largest possible value of $| \delta(A,B)| $ denoted by
\begin{align}
\label{eq:del}
\partial^n:=\max_{A,B \subset [N]} | \delta(A,B)|.	
\end{align}
In general, a low $\partial^n$ guarantees that edge density is relatively homogeneous throughout the graph.


Based on this observation, we distinguish between two types of graph sequences:
\begin{itemize}
	\item $\left(G^n \right)_{n=1}^{\infty}$ is quasi-random if $\lim_{n \to \infty} \partial^n=0.$
	\item $\left(G^n\right)_{n=1}^{\infty}$ is non quasi-random if $\limsup_{n \to \infty}\partial^n>0.$
\end{itemize}
These are the only possibilities. The term quasi-random is motivated by the fact that for certain classes of random graphs, $\partial^n$ will indeed be small. This is addressed in more detail in Section \ref{c:random}.

The following measures of discrepancy will also be utilized:
\begin{align}
\label{eq:del1}
\partial_1^n :=& \max_{A \subset [N]}| \delta(A,A) |,\\
\label{eq:del2}
\partial_2^n :=&\max_{\substack{A,B \subset [N] \\ A \cap B= \emptyset}}|\delta(A,B)|, \\
\label{eq:del*}
\begin{split}
\partial_*^n :=& \max_{A \subset [N]}|\delta(A,[N])|=\max_{A \subset [N]} \left | \frac{1}{N \bar{d}^n} \sum_{i \in A}d^n(i)-\frac{|A|}{N} \right |=\\
&\max_{A \subset [N]} \left | \frac{1}{N \bar{d}^n} \sum_{i \in A}\left[d^n(i)-\bar{d}^{n} \right] \right |.
\end{split}
\end{align}

Intuitively, $\partial_1^n$ measures the worst possible discrepancy within a single set, while $\partial_2^n$ between two disjoint sets. $\partial_*^n$, on the other hand, depends only on the degree sequence of the graph, and measures the concentration of the degree distribution around $\bar{d}^n$. $\partial_*^n=0$ holds if and only if $G^n$ is regular.

Note that
\begin{align}
\label{eq:del*_nice}
&\partial_*^n=\delta \left(V_{+}^n, [N]\right)=-\delta \left(V_{-}^n, [N]\right)=\frac{1}{2N \bar{d}^n} \sum_{i =1}^N \left|d^n(i)-\bar{d}^{n} \right|,
\end{align}
where
\begin{align}
\label{eq:V+}
\begin{split}
&V_{+}^n:= \left\{ \left. i \in [N] \right | d^n(i) \geq \bar{d}^n \right\},\\
& V_{-}^n:= \left\{ \left. i \in [N] \right | d^n(i) < \bar{d}^n \right\}.
\end{split}
\end{align}
The hierarchies between the quantities are stated below.

\begin{lemma}
\label{l:hierarchy}
\begin{align}
\label{eq:hierarchy1}
 \max\left \{\partial_1^n, \partial_2^n \right \} \leq & \partial^n \leq \frac{11}{2}\partial_1^n \\
\label{eq:hierarchy2}
\partial_*^n \leq & \partial^n \\
\label{eq:hiererchy3}
\partial_1^n \leq & \partial_2^n+\partial_*^n
\end{align}
\end{lemma}

According to \eqref{eq:hierarchy1}, it is easy to see that $\partial^n$ and $\partial_1^n$ are equivalent in the sense that either both of them converge to $0$ or neither. Thus $\partial_1^n$ is also appropriate for characterizing whether the sequence $\left(G^n \right)_{n=1}^{\infty} $ is quasi-random or not.

Due to \eqref{eq:hierarchy2}, $\partial_*^n \to 0$ is necessary for the graph sequence to be quasi-random. However, it is not sufficient as the following example shows: $G^n$ is a bipartite graph on $N(n)=2n$ vertices and each vertex having degree $1$. As it is a regular graph, $\partial_*^n=0$ while choosing $A$ as one of the two classes leads to 
$|\delta(A,A)|=0+\left(\frac{1}{2} \right)^2=\frac{1}{4}$, so
$$\limsup_{n \to \infty} \partial^n \geq \frac{1}{4}>0.$$

Finally, from \eqref{eq:hierarchy1} and \eqref{eq:hiererchy3}, under the condition $\partial_*^n \to 0$, $\partial_1^n, \partial_2^n$ and $\partial^n$ are equivalent in the sense that either all of them converge to 0 or none of them.

Another measure of discrepancy, more suited towards spectral theoretical considerations later, is based on the \emph{volume} of a set $A \subset [N]$ defined as
$$ vol(A)=e(A,[N])=\sum_{i \in A }d^n(i). $$

The corresponding discrepancy is then defined as
\begin{align*}
&\tilde{\delta}(A,B):=\frac{e(A,B)}{vol([N])}-\frac{vol(A)}{vol([N])}\frac{vol(B)}{vol([N])}, \\
& \tilde{\partial}^n:=\max_{A,B \subset [N]} \left | \tilde{\delta}(A,B) \right |.
\end{align*}
Note that $vol([N])=e([N],[N])=N \bar{d}^n.$

When the degree distribution is fairly homogeneous, the two quantities do not differ much:
\begin{lemma}
\label{l:vol1}
$\forall A,B \subset [N]:$
\begin{align*}
\left |\delta(A,B)-\tilde{\delta}(A,B) \right | & \leq 2 \partial^n_*, \\
\left|\partial^n -\tilde{\partial}^n \right | & \leq 2 \partial_*^n	.
\end{align*}
\end{lemma}

We also define discrepancy with respect to induced subgraphs. Let $H$ be an induced subgraph of $G$ (identified as a subset of the vertices $H\subset [N]$). Then for any $A,B \subset [N]$,
$$e_H(A,B):=e\left(A \cap H, B \cap H \right).$$

The discrepancy on $H$ is defined as
\begin{align*}
& \tilde{\delta}_H(A,B):= \frac{e_H(A,B)}{vol(H)}-\frac{vol \left(A \cap H \right)}{vol(H)}\frac{vol \left(B \cap H \right)}{vol(H)}, \\
& \tilde{\partial}_H^n:=\max_{A,B \subset [N]}\left|\tilde{\delta}_H(A,B) \right|=\max_{A,B \subset H}\left|\tilde{\delta}_H(A,B) \right|.
\end{align*}
These quantities are insensitive to the structure of $G$ on $H^c$.

When $H$ includes most of the vertices of the original graph, $\tilde{\partial}^n$ and $\tilde{\partial}_H^n$ are close, as formulated rigorously by the following lemma:
\begin{lemma}
\label{l:vol2}
Assume $\frac{vol\left(H^c\right)}{vol([N])} \leq \frac{1}{2}.$ Then $ \forall A,B \subset [N] $
\begin{align*}
\left |\tilde{\delta}(A,B)-\tilde{\delta}_H(A,B) \right | &\leq 10 \frac{vol\left(H^c\right)}{vol([N])},  \\
\left | \tilde{\partial}^n-\tilde{\partial}_H^n \right | &\leq 10 \frac{vol\left(H^c\right)}{vol([N])}.
\end{align*}
\end{lemma}

A similar statement was given in \cite{ModularityER} for a related quantity called modularity.

\subsection{Spectral properties}
\label{c:spectral}

In this section we discuss how discrepancies can be bounded using spectral theory.

We introduce the diagonal degree matrix $\mathcal{D}^n:=\diag \left(d^n(1), \dots, d^n(N) \right)$ and re-scale the adjacency matrix as
\begin{align*}
& \mathcal{B}^n:=\left( \mathcal{D}^n \right)^{-\frac{1}{2}} \mathcal{A}^n \left( \mathcal{D}^n \right)^{-\frac{1}{2}}.
\end{align*}

We order the eigenvalues as $\lambda_1\left( \mathcal{B}^n \right) \geq \dots \geq \lambda_N\left( \mathcal{B}^n \right).$
According to the Perron-Frobenious theorem, all eigenvalues are real, and  $\left | \lambda_i\left(\mathcal{B}^n\right) \right | \leq \lambda_1\left(\mathcal{B}^n\right). $

The second largest eigenvalue in absolute value is denoted by
\begin{align*}
& \lambda^n:=\max \left \{\lambda_2\left(\mathcal{B}^n\right), -\lambda_N\left(\mathcal{B}^n\right) \right \}.
\end{align*}
$1-\lambda^n$ is called the spectral gap. Matrices with large spectral gap are generally ``nice'': they have good connectivity properties, random walks on them converges to equilibrium fast and vertices are well-mixed. The following proposition is a special case of Lemma 1 in \cite{Bolla2011}.

\begin{prop} (Expander mixing lemma)
For all $A,B \subset [N]$
\begin{align}
\label{eq:expander}
\left|\frac{e(A,B)}{vol([N])}-\frac{vol(A)}{vol([N])} \frac{vol(B)}{vol([N])} \right| \leq \lambda^n \sqrt{\frac{vol(A)}{vol([N])}\frac{vol(B)}{vol([N])}}.
\end{align}
\end{prop} 

From \eqref{eq:expander} it is easy to see that
\begin{align}
\label{eq:spectral}
& \tilde{\partial}^n \leq \lambda^n,
\end{align}
meaning a large spectral gap guarantees low discrepancies, at least in the degree biased setting.

When $G^n$ is regular, the expressions simplify to $\tilde{\partial}^n=\partial^n$ and $\lambda^n$ can be expressed as
$$ \lambda^n=\frac{1}{d^n}\max \left \{\lambda_2\left(\mathcal{A}^n\right), -\lambda_N\left(\mathcal{A}^n\right) \right \}. $$

For fixed $d^n=d$, the spectral gap can not be too close to $1$. Based on \cite{Nilli1991}, for every $d$ and $\varepsilon>0$ there exists an $N_0$ such that for every graph with at least $N_0$ vertices
\begin{align}
\label{eq:spectral_lower_bound}
\lambda^n \geq \frac{2 \sqrt{d-1}}{d}-\varepsilon. 
\end{align}

\subsection{Random graphs}
\label{c:random}

In this subsection we discuss the types of random graphs used in this paper and their properties.

The first example is the Erd\H os-R\'enyi graph $\mathcal{G}_{\textrm{ER}}(N,p^n)$, which contains $N$ vertices and each pair of vertices are connected with probability $p^n$, independent from other pairs. The expected average degree is
$$ \langle d \rangle^n :=(N-1)p^n. $$

Another type of random graphs of interest is the random regular graph, denoted by $\mathcal{G}_\textit{reg}(N,d^n).$ It is a random graph chosen uniformly from among all $d^n$-regular graphs on $N$ vertices ($Nd^n$ is assumed to be even).

The bound \eqref{eq:spectral_lower_bound} is sharp for $d$-regular random graphs with even $d\geq 4$ in the following sense \cite{d_reugular_eigenvalue}:
\begin{align}
\label{eq:ramanujan}
& \lambda^n \leq \frac{2 \sqrt{d-1}}{d}+\varepsilon \ \ w.h.p.
\end{align}
Similar results are shown in \cite{d_reugular_eigenvalue} when $d$ is odd.

For diverging degrees  $N^{\alpha} \leq d^n \leq \frac{N}{2}$ for some $0 < \alpha <1$ there is some $C_\alpha$ such that 
\begin{align}
\label{eq:ramanujan2}
& \lambda^n \leq \frac{C_\alpha}{\sqrt{d^n}}
\end{align} 
with probability at least $1-\frac{1}{N}$ \cite{spectralgap_dense_regular}.

The spectral gap of the Erdős-Rényi graph is analogous whit the catch that  $d^n$ should be replaced with $\langle d \rangle ^n$ at least when one constrain ourselves to an appropriate large subgraph. For this purpose we introduce the core of the matrix defined in \cite{LaplacianER}.

The $core\left(G^n\right)$ is a subgraph spanned by $H \subset [N]$ where $H$ is constructed as follows.
\begin{itemize}
	\item Initialize $H$ as the subset of vertices who have at least $\frac{\langle d \rangle ^n}{2}$ vertices.
	\item While there is a vertex in $H$ with eat least $100$ neighbors in $[N] \setminus H$ remove that vertex.
\end{itemize}

According to the following proposition $core\left(G^n\right)$ is a sub graph covering most of the vertices whit a large spectral gap.
\begin{prop}
	\label{p:core}
	Assume $c_0 \leq \langle d \rangle^n \leq 0.99N $ for some sufficiently large $c_0$. Then there is a $c_1$ such that w.h.p.
	\begin{align*}
	& \left|H^c\right| \leq N e^{-\frac{\langle d \rangle^n}{c_0}} \\
	& \lambda_H^n \leq \frac{c_1}{\sqrt{\langle d \rangle^n}},
	\end{align*}
	where $1-\lambda_H^n$ is the spectral gap on the on the subgraph. 
\end{prop} 
The proof of Proposition \ref{p:core} can be found in \cite{LaplacianER}.

When $\langle d \rangle ^n \geq \log^2 N $ a simpler statement can be made based on \cite{ERspectra2}.
\begin{prop}
\label{ERspectra2}
For an $\mathcal{G}_{ER}\left(N,p^n\right)$ graph with $\langle d \rangle ^n \geq \log^2 N$
\begin{align*}
\lambda^n =O_{p} \left(\frac{1}{\sqrt{\langle d \rangle^n}} \right).
\end{align*}	
\end{prop}

\section{Markov processes on graphs}
\label{c:markov}

This section describes the dynamics.

Assume a graph $G^n$ is given. Each vertex is in a state from a finite state space $\S$. $\xi_{i,s}^n(t)$ denotes the indicator that vertex $i$ is in state $s$ at time $t$; the corresponding vector notation is $$\xi_i^n(t)=\left(\xi_{i,s}^n(t)\right)_{s \in \S}.$$

Our main focus is to describe the behavior of the average
$$ \bar{\xi}^n(t):=\frac{1}{N}\sum_{i=1}^N \xi_i^n(t). $$
$\bar{\xi}_s^n(t)$ can be interpreted as the ratio of vertices in state $s \in \S$ at time $t$. It is worth noting that both $\xi_i^n(t)$ and $\bar{\xi}^n(t)$ lie inside the simplex
$$ \Delta^{\S}:= \left \{  v \in [0,1]^{\S}  \left | \sum_{s \in \S}v_s=1 \right.  \right \}. $$

Let
$$ V_s^n(t):=\left \{ i \in [N] \left | \xi_{i,s}^n(t)=1 \right. \right\}$$
denote the set of vertices in state $s$ at time $t$. The normalized number of edges between vertices in state $s$ and $s'$ is denoted by
$$ \nu_{ss'}^n(t):= \frac{e \left(V_s^n(t), V_{s'}^n(t)\right)}{N \bar{d}^n}=\frac{1}{N \bar{d}^n}\sum_{i=1}^N \sum_{j=1}^N a_{ij}^n \xi_{i,s}^n(t)\xi_{j,s'}^n(t). $$
We may also reformulate the ratio $\bar{\xi}_s^n(t)$ as
$$ \bar{\xi}_{s}^n(t)= \frac{\left|V_s^n(t) \right|}{N}. $$

Each vertex may transition to another state in continuous time. The transition rates of a vertex may depend on the number of the neighbors of that vertex in each state. For vertex $i$, the number of its neighbors in state $s$ is
$$\sum_{j=1}^{N}a_{ij}\xi_{j,s}^n(t).$$

We introduce the normalized quantity
$$ \phi_{i,s}^n(t):= \frac{1}{\bar{d}^n}\sum_{j=1}^N a_{ij}\xi_{j,s}^n(t), $$
with corresponding vector notation $\phi_i^n(t):= \left( \phi_{i,s}^n(t)\right)_{s \in \S}$. Note that 
\begin{align}
\label{eq:phi_norm}
 \left \| \phi_{i}^n(t) \right \|_{1}=\frac{1}{\bar{d}^n}\sum_{j=1}^N a_{ij}^n \underbrace{\sum_{s \in \S}\xi_{i,s}^n(t)}_{=1}=\frac{d^n(i)}{\bar{d}^n}.
\end{align}
Typically $\phi_i^n(t) \notin \Delta^{\S}.$

Transition rates are described by the functions $q_{ss'}: \mathbb{R}^{\S} \to \mathbb{R}$. With slightly unconventional notation, $q_{ss'}$ will refer to the transition rate from state $s'$ to $s$. This convention enables us to work with column vectors and multiplying by matrices from the left. The matrix notation $Q(\phi)= \left(q_{ss'}(\phi) \right)_{s,s' \in \S}$ will be used.

We require $q_{ss'}(\phi) \geq 0$ for $s \neq s'$ for non-negative inputs $\phi \geq 0$. For the diagonal components,
$$q_{ss}=-\sum_{s' \neq s}q_{s's}$$
corresponds to the outgoing rate from state $s$.

The dynamics of $\left(\xi_{i}^n(t)\right)_{i=1}^{N}$ is a continuous-time Markov chain with state-space $\S^{N}$ where each vertex performs transitions according to the transition matrix $Q\left(\phi_{i}^n(t) \right)$, independently from the others. After a transition, vertices update their neighborhood vectors $\phi_i^n(t)$. This means that, at least for a single transition, each vertex is affected only by its neighbors. We call such dynamics local-density dependent Markov processes.

Our main assumption is that the rate functions $q_{ss'}$ are affine (linear, also allowing a constant term), meaning there are constants $q_{ss'}^{(0)}, \left(q_{ss',r}^{(1)} \right)_{r \in \S}$ such that
$$ q_{ss'}(\phi)=q_{ss'}^{(0)}+\sum_{r \in \S}q_{ss',r}^{(1)}\phi_r. $$

From the non-negative assumption it follows that these coefficients are non-negative. Let $q_{\max}^{(1)}$ denote the maximum of $\left |q_{ss',r}^{(1)} \right |$. 

From this definition,
\begin{align*}
& \frac{\d}{\d t}\E \left( \left. \xi_i^n(t) \right | \mathcal{F}_t \right)=Q \left( \phi_i^n(t) \right)\xi_i^n(t),
\end{align*}
or, writing it out for each coordinate,
\begin{align}
\label{eq:condexp1}
\begin{split}
\frac{\d}{\d t}\E \left( \left. \xi_{i,s}^n(t) \right | \mathcal{F}_t \right)=& \sum_{s' \in \S}q_{ss'}\left(\phi_{i}^n(t)\right)\xi_{i,s'}^n(t)=\\
&\sum_{s' \in \S}q_{ss'}^{(0)}\xi_{i,s'}^n(t)+\sum_{s' \in \S}\sum_{r \in \S}q_{ss',r}^{(1)}\phi_{i,r}^n(t) \xi_{i,s'}^n(t).
\end{split}
\end{align}
Using the identity
\begin{align*}
& \frac{1}{N}\sum_{i=1}^N \phi_{i,r}^n(t)\xi_{i,s'}^n(t)= \frac{1}{N \bar{d}^n}\sum_{i=1}^N \sum_{j=1}^N a_{ij}^n \xi_{i,s'}^n(t)\xi_{j,r}^n(t)=\nu_{s'r}^n(t),
\end{align*}
after taking the average in \eqref{eq:condexp1} with respect to $i$ we get that the rate at which the ratio $ \bar{\xi}_{s}^n(t)$ changes can be calculated as
\begin{align}
\label{eq:condexp2}
& \frac{\d}{\d t} \E \left( \left. \bar{\xi}_s^n(t) \right | \mathcal{F}_t \right)=\underbrace{ \sum_{s' \in \S}q_{ss'}^{(0)}\bar{\xi}_{s'}^n(t)+\sum_{s' \in \S}\sum_{r \in \S}q_{ss',r}^{(1)}\nu_{s'r}^n(t)}_{:=g_s\left( \bar{\xi}^n(t), \nu^n(t) \right)}.
\end{align}

\subsection*{Examples}

Next we give some examples for Markov processes on graphs, with special focus on epidemiological ones.

Conceptually the easiest epidemiological model is the SIS model. The state space is $\S=\{S,I\}$, $S$ for susceptible and I for infected individuals. The transition rates are
\begin{align*}
q_{SI}(\phi)=&\gamma, \\
q_{IS}(\phi)=&\beta \phi_I,
\end{align*}
meaning susceptible individuals are cured with constant rate and susceptible individuals become infected with rate proportional to the number of their infected neighbors.

The SIR model describes the situation when cured individuals cannot get reinfected. The state space is $\S=\{S,I,R\}$, including $R$ for recovered individuals. The dynamics is modified as
\begin{align*}
	q_{RI}(\phi)=&\gamma, \\
	q_{IS}(\phi)=&\beta \phi_I.
\end{align*}

The SI model describes the situation when there is no cure. This might be a realistic model for the diffusion of information. It is the special case of either SIS or SIR with $\gamma=0$. In this paper, we will regard it as a special case of SIR, allowing for a state R which basically acts as an isolated state. (This approach will be useful for counterexamples.)

For later use, we also introduce notation for terms of order 2 and 3:
\begin{align*}
 [A,B]^n(t):=&\E \left[\sum_{i =1}^N \sum_{j=1}^Na_{ij}^n \xi_{i,A}^n(t)\xi_{j,B}^n(t) \right ]\quad \textrm{and} \\
 [A,B,C]^n(t):=&\E \left[\sum_{i =1}^N \sum_{j=1}^N \sum_{k=1}^Na_{ij}^n a_{jk}^n \xi_{i,A}^n(t)\xi_{j,B}^n(t) \xi_{k,C}^n(t) \right ]
\end{align*} 
denote the expected number of $AB$ pairs and $ABC$ triples. Note that
$$ \nu_{SI}^n(t)=\frac{[SI]^n(t)}{N \bar{d}^n}. $$

 According to Theorem 4.4 in \cite{SimonBook} for the SI process 
\begin{align}
\label{eq:simon}
\frac{\d}{\d t}[SI]^n(t)=\frac{\beta}{\bar{d}^n} \left([SSI]^n(t)-[ISI]^n(t)-[SI]^n(t) \right).	
\end{align}

We also introduce an auxiliary model called the degree process. The state space is $\S=\{a,b\}$ and the only transition, $a\to b$, with rate
\begin{align*}
& q_{ba}(\phi)=\phi_a+\phi_b=\left\| \phi \right \|_1=\frac{d^n(i)}{\bar{d}^n}.
\end{align*}
Since the state of the neighbors does not influence the transition rate, the evolution of the vertices is independent from each other.

\section{Homogeneous mean field approximation}
\label{s:HMFA}

The evolution of the Markov processes introduced in Section  \ref{c:markov}  could be described a system of linear ODEs given by Kolmogorov's forwards equation in principle. However, as the state space is $\mathcal{S}^N $, solving said system is not viable even for relatively small values of $N$. A remedy for this problem is to assume interactions are well mixed so that the dynamics could be described by a few macroscopic averages then derive equations for the reduced system. This is what he homogeneous mean-field approximation  (HMFA) hopes to achieve whit and ODE system whit $|\mathcal{S}|$ variables.     

HMFA is based on the following two assumptions:
\begin{itemize}
	\item Low variance: $\bar{\xi}^n_s(t)$ is close to deterministic when $N$ is large.
	\item The graph is well-mixed, discrepancies are low.
\end{itemize}
We present an intuitive derivation of the governing ODE using these two assumptions. We replace $\nu_{s'r}^n(t)$ by $\bar{\xi}_{s'}^n(t)\bar{\xi}_{r}^n(t)$ in \eqref{eq:condexp2} based on the well-mixed assumption to get
\begin{align*}
& f_s\left(\bar{\xi}_{s}^n(t) \right):=\sum_{s' \in \S}q_{ss'}^{(0)}\bar{\xi}_{s'}^n(t)+\sum_{s' \in \S}\sum_{r \in \S}q_{ss',r}^{(1)}\bar{\xi}_{s'}^n(t)\bar{\xi}_{r}^n(t),
\end{align*}
with corresponding vector notation
$$f\left(\bar{\xi}^n(t)\right)=Q \left( \bar{\xi}^n(t) \right) \bar{\xi}^n(t).$$
$f$ is Lipschitz continuous in $\ell^1$ norm on the simplex $\Delta^{\S}$; its constant is denoted by $L_{f}$.

The error arising from this approximation can be bounded from above by $\partial^n$ due to
\begin{align}
\label{ineq:del}
\begin{split}
\left | \nu_{s'r}^n(t)-\bar{\xi}_{s'}^n(t)\bar{\xi}_{r}^n(t) \right | =&\left| \frac{e \left(V_{s'}^n(t),V_r^n(t) \right)}{N \bar{d}^n}-\frac{\left|V_{s'}^n(t) \right|}{N} \frac{\left|V_{r}^n(t) \right|}{N} \right |\\
=&\left| \delta\left(V_{s'}^n(t),V_{r}^n(t)\right) \right| \leq \partial^n.
\end{split}
\end{align}

Using the low variance assumption, we replace the ratio $\bar{\xi}_s^n(t)$ with a deterministic quantity $u_s(t)$. Based on the second assumption and \eqref{eq:condexp2}, $u(t)$ must satisfy the system of ODEs
\begin{align}
\label{eq:u}
& \frac{\d}{\d t}u(t)=f \left(u(t) \right).
\end{align}

This system of ODEs satisfy the the following existence uniqueness and positivity properties
\begin{lemma}
\label{l:ODE}	
	
Assume $u(0) \in \Delta^{\S}$. Then there is a unique global sollution to \eqref{eq:u} such that $u(t) \in \Delta^{\S}$ for all $t \geq 0.$ 
\end{lemma}.

As the coordinates $\left(u_s(t) \right)_{s \in \mathcal{S}}$ are linearly dependent only $|\mathcal{S}|-1$  ODEs need to be solved in practice.

For the degree process \eqref{eq:u} takes the form 
\begin{align*}
& \frac{\d}{\d t}u_a(t)=-u_a(t),
\end{align*}
while for the SIR process
\begin{align*}
& \frac{\d}{\d t} u_S(t)=-\beta u_I(t)u_{S}(t)\\
& \frac{\d}{\d t}u_I(t)=\beta u_I(t)u_{S}(t)-\gamma u_I(t).
\end{align*}

\begin{definition}
We say the HMFA is \emph{accurate for the graph sequence} $\left(G^n \right)_{n=1}^{\infty}$ if the following holds: we fix any arbitrary linear model $Q: \mathbb{R}^{\S \times \S} \to \mathbb{R}^\S$ and asymptotic initial condition $u(0) \in \Delta^{\S}$. $\left( \xi_{i}^n(0) \right)_{n=1}^{\infty}$ is an arbitrary sequence of initial conditions such that $\bar{\xi}^n(0) \to u(0)$. Then for any $T>0$
$$\lim_{n \to \infty}\sup_{0 \leq t \leq T}\left \| \bar{\xi}^n(t)-u(t) \right \|_1=0 \ \ st. $$ 
Otherwise we say the HMFA is not accurate or inaccurate for the graph sequence.
\end{definition}

Note that we implicitly assumed $N(n) \to \infty$, otherwise $\bar{\xi}^n(t)$ would remain stochastic thus the deterministic approximation would trivially contain some non-vanishing error.

The requirement for the HMFA to work for any linear $Q$ is somewhat restrictive, as there may be cases when the HMFA is accurate for some processes but not for others. Notably for $Q(\phi)=Q^{(0)}$ constant, the vertices are independent, hence HMFA works for any $\left( G^n\right)_{n=1}^{\infty}$ based on the law of large numbers. We wish to exclude these pathological cases by requiring convergence for all $Q$.

Similarly, it may be possible that HMFA works for some sequence of initial conditions but not for others. For example, in the SIS model starting the epidemics with $0$ infected individual results in the same limit for both the exact stochastic process and the ODE, regardless of the graph sequence. It is also possible that the stochastic process exhibits wildly different behavior for different initial conditions $(\xi_i^n(0))_{i=1}^{\infty}$ and $\left( \left(\xi_i^{n}\right)'(0) \right)_{i=1}^n$ while $\bar{\xi}^n(0)$ and $\left(\bar{\xi}^n\right)(0)'$ converge to the same $u(0)$, rendering the ODE unable to distinguish between the two cases.

This can be illustrated by the following example: Let $G^n$ be the star graph with $N(n)=n$ vertices and $i=1$ being the hub. For the SI process if we choose the initial condition with no infection, then $\bar{\xi}^n(t)$ and $u(t)$ will lead to the same conclusion. However, when we only infect $i=1$ leaving the rest susceptible, then $i=1$ will stay in that state forever while the leaves are infected independently with rate $\frac{\beta}{\bar{d}^n}\approx \frac{\beta}{2},$ thus 
$$ \bar{\xi}_{I}^n(t) \to 1-e^{-\frac{\beta}{2}t} \ \ st., $$  
while $ \bar{\xi}_{I}^n(0)=\frac{1}{n} \to 0 =:u_{I}(0).$

In general, for non quasi-random graph sequences, the initial condition can be selected in a similar manner to conclude that HMFA is not accurate.


\section{Results}
\label{s:results}

The central claim of this paper is the following:
\begin{theorem}[Main]
	\label{t:main}
	
For a graph sequence $(G^n)_{n=1}^{\infty}$ the HMFA is accurate if and only if said sequence is quasi-random.
\end{theorem}

The following theorem shows the $\Leftarrow$ direction of Theorem \ref{t:main} and provides a quantitative error bound which can be used for non-quasi-random graph-sequences as well or even concrete graphs.
\begin{theorem}
\label{t:positive}

For all $T>0$
\begin{align*}
& \sup_{0 \leq t \leq T} \left \|\bar{\xi}^n(t)-u(t) \right \|_1 \leq \left(  \left \|\bar{\xi}^n(0)-u(0) \right \|_1+O_{p} \left( \frac{1}{\sqrt{N}} \right)+C(T)\partial^n \right)e^{L_f T},
\end{align*}
where $C(T)=q_{\max}^{(1)}\left | \S \right|^3 T.$
\end{theorem}
The first two terms are vanishing as $N(n) \to \infty$ and $\bar{\xi}^n(0) \to u(0)$ so the only nontrivial part is $O(\partial^n)$, which goes to $0$ when the graph sequence is quasi-random by definition. The term vanishes when $q_{\max}^{(1)}=0$ since in that case $Q(\phi)=Q^{(0)}$ is constant making vertices independent. It is worth mentioning that based on the proof of Lemma \ref{l:fluctuations}, $O_p\left(\frac{1}{\sqrt{N}} \right)$ only depends on $q_{\max}^{(0)},q_{\max}^{(1)},T$ and apart form $N(n)$, it is independent from the graph sequence.

For some random graphs we may bound $\partial^n$ with something as the same order as $\left(\bar{d}^n \right)^{-\frac{1}{2}}.$

\begin{theorem}[Discrepancy bound for Erd\H os--R\'enyi graphs]
\label{t:ER}

For a $\mathcal{G}_{\textrm{ER}}(N,p^n)$  Erdős-Rényi graphs sequence with $ \langle d \rangle^n \to \infty $ 
\begin{align}
\label{eq:ER_partial_bound}
& \partial^n=O_p \left( \frac{1}{\sqrt{\langle d \rangle^n }} \right).
\end{align}
\end{theorem}

Similar results can be said about $\mathcal{G}_{\textit{reg}}(N,d^n)$ random regular graphs. Based on \eqref{eq:spectral},\eqref{eq:ramanujan} and \eqref{eq:ramanujan2}
\begin{align*}
& \partial^n=\tilde{\partial}^n \leq \lambda^n=O_p \left( \frac{1}{\sqrt{d^n}} \right),
\end{align*}
when the appropriate conditions hold.

The next theorem shows the $\Rightarrow$ direction for Theorem \ref{t:main}.
\begin{theorem}
\label{t:negative}
Let $\left(G^n\right)_{n=1}^{\infty}$ a non quasi-random graph sequence. Then the HMFA is not accurate for $\left(G^n\right)_{n=1}^{\infty}$.
\end{theorem}

Lastly, we introduce some graphs that are not quasi-random.

\begin{theorem}
\label{t:not_quasi_random}
The graph sequence $\left(G^n\right)_{n=1}^{\infty}$ is not quasi-random if at least one of the statement below is true:
\begin{enumerate}
	\item[1)] $G^n$ is bipartite for infinitely many $n$,
	\item[2)] $\limsup_{n \to \infty}\theta^n>0$, meaning, the giant component does not cover most of the vertices,
	\item[3)] $\limsup_{n \to \infty} \frac{\alpha^n}{N}>0$ (there are large independent sets),
	\item[4)] $\liminf_{n \to \infty} \bar{d}^n < \infty$ ($\bar{d}^n$ does not approach infinite).
\end{enumerate}
\end{theorem}
Interestingly, based on Proposition \ref{t:not_quasi_random}, not even well known random graphs such as the Erd\H os-R\'enyi graph and random $d$-regular graphs are quasi-random when the degrees are bounded. Combined with Theorem \ref{t:ER}, an Erd\H os-R\'enyi graph will be quasi-random if and only if $\langle d \rangle^n \to \infty. $

\section{Proofs}
\label{s:proof}

\subsection{Smaller statements}

In this section we are giving the proofs for Lemma \ref{l:hierarchy}-\ref{l:ODE}.

\begin{proof}(Lemma \ref{l:hierarchy})
		
$\partial_*^n, \partial_1^n, \partial_2^n \leq \partial^n$ as the maximum is taken on a wider domain in $\partial^n$. This immediately shows \eqref{eq:hierarchy2} and the left hand side of \eqref{eq:hierarchy1}. For the right hand side, let $\tilde{A},\tilde{B} \subset [N]$ be two disjoint sets.
\begin{align}
\nonumber
& \delta(\tilde{A} \sqcup \tilde{B},\tilde{A} \sqcup \tilde{B})=\delta(\tilde{A},\tilde{A})+\delta(\tilde{B},\tilde{B})+2 \delta(\tilde{A},\tilde{B}) \\
\label{eq:disjoint}
& 2|\delta(\tilde{A},\tilde{B})| \leq |\delta(\tilde{A} \sqcup \tilde{B},\tilde{A} \sqcup \tilde{B})|+|\delta(\tilde{A},\tilde{A})|+|\delta(\tilde{B},\tilde{B})| \leq 3 \partial_1^n
\end{align}

Let $A,B \subset [N]$ be two extreme sets maximizing $\left|\delta(A,B) \right|.$ They can be decomposed into the disjoint sets:
\begin{align*}
\tilde{A}:=& A \setminus B, \\
\tilde{B}:=& B \setminus A, \\
\tilde{C}:=& A \cap B.
\end{align*}

\begin{align*}
\partial^n=& \left|\delta(A,B) \right| =\left| \delta(\tilde{A} \sqcup \tilde{C}, \tilde{B} \sqcup \tilde{C}) \right| \leq \\
& \left|\delta(\tilde{A},\tilde{B}) \right|+\left|\delta(\tilde{A},\tilde{C}) \right|+\left|\delta(\tilde{B},\tilde{C}) \right|+\left|\delta(\tilde{C},\tilde{C}) \right| \overset{\eqref{eq:disjoint}}{\leq}\\
& 3 \cdot \frac{3}{2} \partial_1^n+\partial_1^n=\frac{11}{2}\partial_1^n
\end{align*} 

As for \eqref{eq:hiererchy3} Let $A \subset [N]$ such that $|\delta(A,A)|=\partial_1^n.$
\begin{align*}
& \delta(A,[N])=\delta(A,A)+\delta(A,[N]\setminus A) \\
& \partial_1^N=|\delta(A,A)| \leq |\delta(A,[N]\setminus A)|+|\delta(A,[N])| \leq \partial_2^n+\partial_*^n
\end{align*}
\end{proof}

\begin{proof}(Lemma \ref{l:vol1})

For the first inequality,	
\begin{align*}	
& \left|\delta(A,B)-\tilde{\delta}(A,B) \right|= \left| \frac{vol(A)}{vol([N])}\frac{vol(B)}{vol([N])}-\frac{|A|}{N}\frac{|B|}{N} \right | \leq \\
& \left | \frac{vol(A)}{vol([N])}-\frac{|A|}{N} \right |+\left | \frac{vol(B)}{vol([N])}-\frac{|B|}{N} \right |=\left|\delta(A,[N]) \right|+\left|\delta(B,[N]) \right| \leq 2 \partial_*^n.
\end{align*}

For the second inequality, choose $A,B \subset [N]$ to be extremes sets reaching the value the maximum for $\left|\delta(A,B) \right|$.
\begin{align*}
& \partial^n=\left| \delta(A,B) \right | \leq 2 \partial_*^n+\left| \tilde{\delta}(A,B) \right | \leq 2 \partial_*^n+\tilde{\partial}^n.
\end{align*}
Similarly, when $A,B \subset [N]$ maximizes $\left |\tilde{\delta}(A,B) \right |,$ then
\begin{align*}
	& \tilde\partial^n=\left| \tilde{\delta}(A,B) \right | \leq 2 \partial_*^n+\left| \delta(A,B) \right | \leq 2 \partial_*^n+\partial^n.
\end{align*}
\end{proof}

\begin{proof}(Lemma \ref{l:vol2})
		
For the first term in $\tilde{\delta}(A,B)$ and $\tilde{\delta}_H(A,B)$:
\begin{align*}
 &\frac{e(A,B)}{vol([N])}=\frac{vol(H)}{vol([N])}\frac{e_H(A,B)}{vol(H)}+\\
&\frac{e\left(A \cap H^c,B \cap H\right)}{vol([N])}+\frac{e\left(A \cap H,B \cap H^c\right)}{vol([N])}+\frac{e\left(A \cap H^c,B \cap H^c\right)}{vol([N])} \\
& \left|\frac{e(A,B)}{vol([N])}-\frac{e_h(A,B)}{vol(H)} \right| \leq 4 \frac{vol\left(H^c \right)}{vol(N)}.
\end{align*}
For the second term
\begin{align*}
& \left| \frac{vol(A)}{vol([N])}\frac{vol(B)}{vol([N])}-\frac{vol\left(A \cap H \right)}{vol(H)} \frac{vol\left(B \cap H \right)}{vol(H)} \right| \leq \\
& \left| \frac{vol(A)}{vol([N])}\frac{vol(B)}{vol([N])}-\frac{vol\left(A  \right)}{vol(H)} \frac{vol\left(B  \right)}{vol(H)} \right| +\\
&\left| \frac{vol(A)}{vol(H)}\frac{vol(B)}{vol(H)}-\frac{vol\left(A \cap H \right)}{vol(H)} \frac{vol\left(B \cap H \right)}{vol(H)} \right| \leq\\
& \left|\frac{vol \left(A  \right)}{vol(H)}-\frac{vol \left(A  \right)}{vol([N])} \right|+\left|\frac{vol \left(B  \right)}{vol(H)}-\frac{vol \left(B  \right)}{vol([N])} \right| + \\
& \left| \frac{vol(A)}{vol([N])}-\frac{vol\left(A  \cap H\right)}{vol([N])} \right|+\left| \frac{vol(B)}{vol([N])}-\frac{vol\left(B  \cap H\right)}{vol([N])} \right|\leq\\
& 2 \left| \frac{vol([N])}{vol(H)}-1 \right|+2 \frac{vol\left(H^c\right)}{vol([N])} = 2 \left| \frac{1}{1-\frac{vol(H^c)}{vol([N])}}-1 \right|+2 \frac{vol\left(H^c\right)}{vol([N])} \leq 6 \frac{vol(H^c)}{vol([N])},
\end{align*}
where we used $\frac{1}{1-x}-1=\frac{x}{1-x} \leq 2x$ whit $x=\frac{vol(H^c)}{vol([N])} \leq \frac{1}{2}$ in the last step.

Putting together the two bounds yields
\begin{align*}
& \left|\tilde{\delta}(A,B)-\tilde{\delta}_H(A,B) \right| \leq 10 \frac{vol(H^c)}{vol([N])}.
\end{align*}

The second inequality can be proven in the same fashion as in Lemma \ref{l:vol1} by choosing $A,B \subset [N]$ that maximizes $\left|\tilde{\delta}(A,B) \right|$ then the same whit $\left|\tilde{\delta}_H(A,B) \right|.$
\end{proof}

\begin{proof}(Lemma \ref{l:ODE})
	
For technical reasons we modify the ODE system slightly. Introduce
\begin{align*}
	\hat{q}_{ss'}\left( \phi  \right):=&q_{ss'}^{(0)}+\sum_{r \in \S}q_{ss',r}^{(1)}\left| \phi_{r} \right| \ \ (s' \neq s) \\
	\hat{q}_{ss}\left( \phi  \right):=&\sum_{s' \neq s}q_{s's}(\phi).
\end{align*}
This modification ensures $\hat{q}_{ss'}(\phi)\geq 0$ for all values of $\phi$ not just the non negative ones. Note that $\left.\hat{q}_{ss'}\right|_{\phi \geq 0}=\left.q_{ss'}\right|_{\phi \geq 0}.$

 With $\hat{Q}(\phi)=\left(\hat{q}_{ss'}(\phi) \right)_{s,s' \in \mathcal{S}}$ the altered ODE system is defined as
\begin{align}
\label{eq:ODE2}
\frac{\d}{\d t} \hat{u}(t)=\hat{Q}(\hat{u}(t))\hat{u}(t)=:\hat{f}(\hat{u}(t)).	
\end{align}

Since the right hand side is locally Lipschitz a unique local solution exists. This solution is either global, or there is a blow up at some time $t_0$. Indirectly assume the latter.

Introduce an auxiliary time inhomogenous Markov process defined on $[0,t_0[.$ The process makes transitions acording to the transition matrix $\hat{Q}(\hat{u}(t))$ where we think of $\hat{u}(t)$ as a known function. $p_s(t)$ is the probability of the auxiliary process is in state $s \in S$ at time $t$. $p(t):= \left(p_{s}(t) \right)_{s \in \S} \in \Delta^{S}$ as $p(t)$ is a distribution on $\S$.

The Kolmogorov equations for the auxiliary process takes the form
\begin{align*}
\frac{\d}{\d}p(t)=&\hat{Q}(\hat{u}(t))p(t).	
\end{align*}	
This enables us to use a Grönwall-type argument.
\begin{align*}
\left \|\hat{u}(t)-p(t) \right \|_1 \leq& \left \|\hat{u}(0)-p(0) \right \|_1 + \int_{0}^{t} \left \| \hat{Q}(\hat{u}(\tau)) [\hat{u}(\tau)-p(\tau)] \right \|_1 \d \tau \leq \\
&\left \|\hat{u}(0)-p(0) \right \|_1 + \max_{0 \leq \tau \leq t} \left \| \hat{Q}(\hat{u}(\tau))  \right \|_1 \int_{0}^{t} \left \| \hat{u}(\tau)-p(\tau) \right \|_1 \d \tau \\
\sup_{0 \leq \tau \leq t}\left \|\hat{u}(\tau)-p(\tau) \right \|_1 \leq& \left \|\hat{u}(0)-p(0) \right \|_1 \exp \left(\max_{0 \leq \tau \leq t} \left \| \hat{Q}(\hat{u}(\tau))  \right \|_1 \cdot t \right)
\end{align*}
Therefore, when the initial conditions are $p(0)=\hat{u}(0)=u(0)$ we have $\hat{u}(t)=p(t)$ for $t \in [0,t_0[,$ implying $\hat{u}(t)$ being bounded which contradicts $\hat{u}(t)$ exploding at time $t_0$. Thus, the solution is global. Also, as $p(t)$ is a probability vector, the coordinates must be non negative and on the simplex $\Delta^{\S}$. Since, for non-negative inputs $\hat{f}\left(\hat{u}(t)\right)=f \left(\hat{u}(t) \right),$ $\hat{u}(t)$ is also a solution to \eqref{eq:u}. Local Lipschitz-continouity of \eqref{eq:u} implies there is no other solution. 
\end{proof}	

\subsection{Positive results}

Int this section Theorem \ref{t:positive} and \ref{t:ER} are proven.

We start by proving Theorem \ref{t:positive}. 

The way we are going to achieve that is by decomposing the approximation error into two parts: fluctuation error and topological error, and then applying Grönwall's inequality to mitigate error propagation.

The fluctuation error in state $s \in \S$ is denoted by $U_s^n(t)$ and defined as
$$ U_s^n(t):=\bar{\xi}_s^n(t)-\E \left(\bar{\xi}_s^n(t) \left | \mathcal{F}_t \right. \right).$$
Note that due to \eqref{eq:condexp2} the conditional expectation can be written as
$$\E \left(\bar{\xi}_s^n(t) \left | \mathcal{F}_t \right. \right)=\bar{\xi}^n_s(0)+\int_{0}^{t}g_s \left(\bar{\xi}^n(\tau), \nu^n(\tau) \right) \d \tau. $$

The topological error denoted by $K_s^n(t)$ is defined as 
$$K_s^n(t):=\int_{0}^{t}g_s \left(\bar{\xi}^n(\tau), \nu^n(\tau) \right)-f_s \left( \bar{\xi}^n(\tau)\right) \d \tau . $$
Using the vector notations $U^n(t)= \left( U_s^n(t) \right)_{s \in \S}, K^n(t)= \left( K_s^n(t) \right)_{s \in \S} $ the ratio vector $\bar{\xi}^n(t)$ can be given as
\begin{align*}
& \bar{\xi}^n(t)=\bar{\xi}^n(0)+U^n(t)+K^n(t)+\int_{0}^{t} f \left( \bar{\xi}^n(\tau) \right) \d \tau.
\end{align*}

\begin{lemma} (Grönwall type error bound)
\label{l:grönwall}

For all $T>0$
\begin{align*}
\sup_{0 \leq t \leq T} \left \|\bar{\xi}^n(t)-u(t) \right \|_1 \leq \left( \left \|\bar{\xi}^n(0)-u(0) \right \|_1+ \sup_{0 \leq t \leq T} \left \| U^n(t) \right \|_1 +C(T)\partial^n \right)e^{L_f T},	
\end{align*}
where $C(T):= q_{\max}^{(1)}\left| \S \right|^3 T$.
\end{lemma}

\begin{proof}(Lemma \ref{l:grönwall})

The integral form of \eqref{eq:u} is
\begin{align*}
& u(t)=u(0)+ \int_{0}^{t} f(u(\tau)) \tau.
\end{align*}
Therefore using Grönwall's inequality one gets
\begin{align*}
\left \| \bar{\xi}^n(t)-u(t) \right \|_1 \leq &\left \| \bar{\xi}^n(0)-u(0) \right \|_1 + \sup_{0 \leq t \leq T} \left \|U^n(t) \right \|_1+\sup_{0 \leq t \leq T} \left \|K^n(t) \right \|_1+\\
&L_f \int_{0}^{t} \left \| \bar{\xi}^n(\tau)-u(\tau) \right \|_1 \d \tau \\
\sup_{0 \leq t \leq T}\left \| \bar{\xi}^n(t)-u(t) \right \|_1 \leq & \left( \left \| \bar{\xi}^n(0)-u(0) \right \|_1 + \sup_{0 \leq t \leq T} \left \|U^n(t) \right \|_1+ \right. \\
&\left .\sup_{0 \leq t \leq T} \left \|K^n(t) \right \|_1 \right)e^{L_f T}.
\end{align*}
What is left to do is to upper-bound the topological error term.

Note that $g_s$ and $f_s$ only differ in their quadratic terms. Using \eqref{ineq:del} yields
\begin{align*}
\left |K_s^n(t) \right | \leq & \int_{0}^{t} \left |g_s \left( \bar{\xi}_s^n(\tau), \nu^n (\tau) \right)-f \left( \bar{\xi}^n(\tau) \right) \right | \d \tau \leq \\
&\int_{0}^{t} \sum_{s' \in \S} \sum_{r \in \S} \left|q_{ss',r}^{(1)} \right| \left| \nu_{s'r}^n(\tau)-\bar{\xi}_{s'}^n(\tau)\bar{\xi}_{r}^n(\tau) \right| \d \tau \leq q_{\max}^{(1)} \left | \S \right |^2 t \partial^n \\
\sup_{0 \leq t \leq T}\left \|K^n(t) \right \|_1 \leq & q_{\max}^{(1)} \left | \S \right |^3 T \partial^n=:C(T)\partial^n.
\end{align*}
\end{proof}

The last step in proving Theorem \ref{t:positive} is to show an appropriate upper-bound for the fluctuation error.

\begin{lemma}(Upper-bounding fluctuations)
\label{l:fluctuations}

For all $T>0$
\begin{align*}
& \sup_{0 \leq t \leq T} \left \|U^n(t) \right \|_{1}=O_{p} \left( \frac{1}{\sqrt{N}} \right).
\end{align*}	
\end{lemma}

\begin{proof}(Lemma \ref{l:fluctuations} )

We represent the ratios $\bar{\xi}_{s}^{n}(t)$ with the help of Poisson-processes similarly as in \cite{Kurtz1978}. For each $s,s' \in S$ let $\mathcal{N}_{ss'}^n(t)$ be independent Poisson-processes with rate $1$ representing a transition from state $s'$ to $s$. Then $\bar{\xi}_s^n(t)$ can be written in distribution as
\begin{align}
\label{eq:Poisson}
\begin{split}
\bar{\xi}_s^n(t)\overset{d}{=}\bar{\xi}_s^n(0)&+\underbrace{ \frac{1}{N} \sum_{s' \neq s} \mathcal{N}_{ss'}^n \left(\int_{0}^{t} \sum_{i=1}^{N} q_{ss'} \left( \phi_i^n(\tau)\right)\xi_{i,s'}^n(\tau) \d \tau  \right) }_{\textit{vertices coming to $s$}} \\
&-\underbrace{ \frac{1}{N} \sum_{s' \neq s} \mathcal{N}_{ss'}^n \left(\int_{0}^{t} \sum_{i=1}^{N} q_{s's} \left( \phi_i^n(\tau)\right)\xi_{i,s}^n(\tau) \d \tau  \right) }_{\textit{vertices leaving $s$}}.
\end{split}
\end{align}
To understand \eqref{eq:Poisson}, the term $\frac{1}{N}$ comes from coming (leaving) vertices increases (decreases) $\bar{\xi}_s^n(t)$ by $\frac{1}{N}$. If the $i$th vertex is at state $s'$, then it makes a transition to state $s$ with rate $q_{ss'}\left( \phi_i^n(t)\right)$. In total, vertices vertices comeing from state $s'$ to $s$ with rate $\sum_{i=1}^{N}q_{ss'}\left( \phi_i^n(t)\right)\xi_{i,s'}^n(t).$

The difference between the Poisson processes and their conditional expectation is denoted by
\begin{align*}
U_{ss'}^n(t):=&\frac{1}{N}\mathcal{N}_{ss'}^n \left(\int_{0}^{t} \sum_{i=1}^{N} q_{ss'} \left( \phi_i^n(\tau)\right)\xi_{i,s'}^n(\tau) \d \tau  \right)-\\
&\int_{0}^{t} \frac{1}{N}\sum_{i=1}^{N} q_{ss'} \left( \phi_i^n(\tau)\right)\xi_{i,s'}^n(\tau) \d \tau.
\end{align*}
Observe $U_{s}^n(t) \overset{d}{=} \sum_{s' \neq s}U_{ss'}^n(t)-U_{s's}^n(t).$ As there are finitely many terms, it is enough to show $\sup_{0 \leq t \leq T} \left| U_{ss'}^n(t) \right |=O_{p}\left(\frac{1}{\sqrt{N}}\right).$

Using \eqref{eq:phi_norm} the rates can be upper-bounded by
\begin{align*}
& \left | \frac{1}{N}\sum_{i=1}^{N} q_{ss'} \left( \phi_i^n(\tau)\right)\xi_{i,s'}^n(\tau) \right | \leq  \frac{1}{N}\sum_{i=1}^{N} \left | q_{ss'} \left( \phi_i^n(\tau)\right) \right |=\\
&\frac{1}{N}\sum_{i=1}^{N} \left | q_{ss'}^{(0)}+\sum_{r \in \S }q_{ss',r}^{(1)}  \phi_{i,r}^n(\tau) \right | \leq \frac{1}{N}\sum_{i=1}^{N} \left(q_{\max}^{(0)}+q_{\max}^{(1)} \left \| \phi_{i}^n(\tau) \right \|_1 \right) =\\
& \frac{1}{N}\sum_{i=1}^{N} \left(q_{\max}^{(0)}+q_{\max}^{(1)} \frac{d^n(i)}{\bar{d}^n} \right)=q_{\max}^{(0)}+q_{\max}^{(1)}=:C.
\end{align*}

As $\mathcal{N}_{ss'}^{n}(Nt)-Nt$ is a martingale, $\left(\mathcal{N}_{ss'}^{n}(Nt)-Nt \right)^2$ is a submartingale, thus from Doob's inequality we get
\begin{align*}
& \pr \left( \sup_{0 \leq t \leq T } \left|U_{ss'}^n(t) \right| \geq \varepsilon \right) \leq \pr \left( \sup_{0 \leq t \leq CT} \left| \mathcal{N}_{ss'}^n(Nt)-Nt \right | \geq N \varepsilon \right)= \\
& \pr \left( \sup_{0 \leq t \leq CT} \left( \mathcal{N}_{ss'}^n(Nt)-Nt \right )^2 \geq N^2 \varepsilon^2 \right) \leq \frac{1}{N^2 \varepsilon^2} \E  \left[\left(\mathcal{N}_{ss'}^n(NCT)-NCT \right)^2 \right]=\\
& \frac{1}{N^2 \varepsilon^2} \mathbb{D}^2 \left( \mathcal{N}_{ss'}^n(NCT)\right)=\frac{CT}{N \varepsilon^2} 
\end{align*}.

We can choose $\varepsilon= \frac{M}{\sqrt{N}}$ to conclude
\begin{align*}
& \pr \left( \sup_{0 \leq t \leq T } \left|U_{ss'}^n(t) \right| \geq \frac{M}{\sqrt{N}} \right) \leq \frac{CT}{M^2} \to 0 \ \ (M \to \infty) \\
& \sup_{0 \leq t \leq T } \left|U_{ss'}^n(t) \right|=O_{p} \left( \frac{1}{\sqrt{N}} \right).
\end{align*}
\end{proof}

Now we can turn to proving the error bound for the Erdős Rényi graph sequence. First, we show that for such graph sequences have close to homogeneous degree distributions.

\begin{lemma}
\label{l:ER}
Let $\left( G^n\right)_{n=1}^{\infty}$ be a sequence of Erdős Rényi graph such that $ \langle d \rangle^n \to \infty . $  Then
\begin{align*}
& \partial_*^n=O_p \left(\frac{1}{\sqrt{\langle d \rangle ^n}}\right).
\end{align*}
\end{lemma}

\begin{proof}(Lemma \ref{l:ER})
Firstly, we modify $\partial^n_*$ such that we replace $\bar{d}^n$ with $\langle d \rangle^n.$
\begin{align*}
& \bar{\partial}_*^n:= \frac{1}{2 N \langle d \rangle^n}\sum_{i=1}^N \left|d^n(i)-\langle d \rangle^n \right|
\end{align*}
$\bar{\partial}_*^n=O_p \left(\frac{1}{\sqrt{\langle d \rangle ^n}}\right) $ as
\begin{align*}
 \E \left( \bar{\partial}_*^n \right)=& \frac{1}{2\langle d \rangle ^n}\E \left(\left| d^n(1)-\langle d \rangle ^n \right| \right) \leq \frac{1}{2\langle d \rangle ^n} \mathbb{D}\left(d^n(1)\right) =\\
 &\frac{\sqrt{(N-1)p^n\left(1-p^n\right)}}{2\langle d \rangle ^n} \leq \frac{1}{2\sqrt{\langle d \rangle^n}}.
\end{align*}

Secondly, $\frac{\bar{d}^n}{\langle d \rangle ^n}=1+O_p \left(\frac{1}{\sqrt{N}}\right)$	as $\E \left( \frac{\bar{d}^n}{\langle d \rangle ^n}\right)=1$ and
\begin{align*}
\mathbb{D}^2 \left( \frac{\bar{d}^n}{\langle d \rangle ^n}\right)=& \frac{1}{N^2 \left(\langle d \rangle ^n \right)^2} \sum_{i=1}^N \sum_{j=1}^N \cov \left(d^n(i), d^n(j)\right)=\\
&\frac{N}{N^2 \left(\langle d \rangle ^n \right)^2} \mathbb{D}^2 \left(d^n(1) \right)+\frac{{N \choose 2}}{N^2 \left(\langle d \rangle ^n \right)^2} \cov\left(d^n(1),d^n(2)\right)  \leq \\
& \frac{(N-1)p^n\left(1-p^n \right)}{N \left(\langle d \rangle ^n \right)^2} +\frac{1}{ \left(\langle d \rangle ^n \right)^2} \mathbb{D}^2 \left(a_{12}^n \right) \leq \\
& \frac{1}{N \langle d \rangle ^n } +\frac{p^n \left(1-p^n \right)}{ \left(\langle d \rangle ^n \right)^2}  \leq \frac{1}{N \langle d \rangle ^n }+\frac{1}{(N-1) \langle d \rangle ^n } \leq \frac{2}{(N-1)},
\end{align*}
where in the last step we used $\langle d \rangle ^n \geq 1 $ for large enough $n$.

Therefore,
\begin{align*}
\partial_*^n=\frac{1}{2N \bar{d}^n}\sum_{i=1}^n \left|d^n(i)-\bar{d}^n \right| \leq \frac{1}{2} \left| \frac{\langle d \rangle^n}{\bar{d}^n}-1 \right|+\frac{\langle d \rangle^n}{\bar{d}^n}\bar{\partial}_*^n=O_p \left(\frac{1}{\sqrt{\langle d \rangle^n}} \right).
\end{align*}	
\end{proof}

\begin{proof}(Theorem \ref{t:ER})

Let $c_0$ and $H=core(G^n)$ be the constant and the core from Proposition \ref{p:core}. For large enough $n$ $c_0 \leq \langle d \rangle^n $ holds based on the assumption.

In the first case, assume $c_0 \leq \langle d \rangle^n \leq 0.99 N.$ then based on Proposition \ref{p:core}  $ \frac{\left|H^c \right|}{N} \leq e^{-\frac{\langle d \rangle^n}{c_0}}$ and w.h.p. and $\lambda_H^n \leq \frac{c_0}{\sqrt{\langle d \rangle^n}}$ w.h.p. Using \eqref{eq:spectral}
$$ \tilde{\partial}_H^n \leq \lambda_H^n=O_p \left(\frac{1}{\sqrt{\langle d \rangle ^n}}\right). $$

By Lemma \ref{l:ER}
\begin{align*}
& \frac{vol\left(H^c \right)}{vol([N])} \leq \frac{\left| H^c \right|}{N}+\partial_*^n \leq e^{-\frac{\langle d \rangle^n}{c_0}}+O_{p} \left(\frac{1}{\sqrt{\langle d \rangle^n}} \right) = O_{p} \left(\frac{1}{\sqrt{\langle d \rangle^n}} \right) \leq \frac{1}{2}
\end{align*}
for large enough $n$ with high probability, hence according to Lemma \ref{l:vol2}
\begin{align*}
\tilde{\partial}^n \leq \tilde{\partial}_H^n+10\frac{vol\left(H^c \right)}{vol([N])}= O_{p} \left(\frac{1}{\sqrt{\langle d \rangle^n}} \right).
\end{align*}

In the second case $\langle d \rangle^n \geq 0.99N \geq \log^2 N.$ Proposition \ref{ERspectra2} ensures
\begin{align*}
	\tilde{\partial}^n \leq \lambda^n =O_{p} \left(\frac{1}{\sqrt{\langle d \rangle^n}} \right).
\end{align*}
in this case as well.

Lastly, applying Lemma \ref{l:vol1} yields
\begin{align*}
\partial^n \leq \tilde{\partial}^n+2 \partial_*^n =O_{p} \left(\frac{1}{\sqrt{\langle d \rangle^n}} \right).	
\end{align*}
\end{proof}	
\subsection{Negative results}

In this section Theorem \ref{t:negative} and \ref{t:not_quasi_random} are proven.

The main idea for Theorem \ref{t:negative} is the following:
Instead of the concrete values of $\bar{\xi}^n(t)$ we are going to work with the expected value
\begin{align*}
& \mu^n(t):=\E \left(\bar{\xi}^n(t)\right).
\end{align*}

By standard arguments one can show that $\sup_{0 \leq t \leq T}\left \|\bar{\xi}^n(t)-u(t) \right \|_1 \to 0 \ \ st.$ implies $\left |\mu_s^n(t)-u_s(t) \right | \to 0 \ \ st.$ for all $0\leq t \leq T$ and $s \in S$ as the quantities in question are uniformly bounded, hence,  it is enough to disprove this simpler statement. 

According to \eqref{eq:condexp2}, $\frac{\d}{\d t} \mu^n(t)$ and $\frac{\d}{\d t}u(t)$  differs  in the quadratic term and those differences can be described by discrepancies. When $\partial^n$ does not vanish we can choose the initial conditions in such a way that the discrepancies are high resulting in a different rates. Since $\mu_s^n(t)-u_s(t) \approx \left( \frac{\d}{\d t}\mu_s^n(0)-\frac{\d}{\d t}u_s(0)\right)t$ this would leave a difference between the prediction of the ODE and the expectation for small values of $t$ resulting in the desired contradiction.

In order for this argument to be valid, we require  the second order terms to be negligible or in other words some regularity in the second derivative. For the ODE it automatically holds, however, for the stochastic processes some problem can arise when the degrees have extreme variation.

To illustrate this phenomena, let $G^n$ a sequence of star graphs and examine the SI process on it. Note that $\bar{d}^n=2\left(1-\frac{1}{N} \right).$ We initially choose the leafs to be infected and keep the center susceptible. This would lead to the large discrepancy
\begin{align*}
\delta\left(V_{S}^n(0),V_{I}^n(0)\right)=\frac{N-1}{2N\left(1-\frac{1}{N}\right)}-\frac{N-1}{N}\frac{1}{N}\approx \frac{1}{2},
\end{align*}
therefore, based on the argument above, one might expect $\mu_{I}^n(t)$ to differ considerably from $u_{I}(t)$. This is not the case thou. Clearly $\bar{\xi}^n_{I}(0)=\frac{N-1}{N} \to 1=:u_{I}(0)$, so the ODE approximation is the constant $u(t) \equiv 1$. The expected value will be close to $0$ as well as 
$$ \frac{N-1}{N}=\mu_{I}^n(0) \leq \mu_{I}^n(t) \leq 1. $$

To mitigate this problem, we will take the following rout:
When $\partial^n_*$ does not converge to $0$, then we show that the HMFA is not accurate for the degree process. When $\partial^n_* \to 0,$ on the other hand, we can thin our extreme sets whit large discrepancy by removing vertices whit too large degrees. The thinned sets will still contain vertices whit high discrepancy, while the second derivatives renaming bounded.

For technical reasons, we will utilize function $u^n(t)$, which is the solution of \eqref{eq:u} with the error-free initial condition $u^n(0)=\bar{\xi}^n(0).$

\begin{lemma}(Technical lemma for counterexamples)
\label{l:techical}

Assume there is a sequence of initial conditions such that there are $s \in \S, K,t_0 > 0$ and a sub-sequence $\left(n_k\right)_{k=1}^{\infty}$ such that
\begin{align}
\label{eq:tech1}
& \forall k \in \mathbb{Z}^{+} \ \sup_{0 \leq t \leq t_0} \left | \frac{\d^2}{\d t^2}\mu_{s}^{n_k}(t) \right | \leq K, \\
\label{eq:tech2}
& \limsup_{k \to \infty} \left | \frac{\d}{\d t}\mu_{s}^{n_k}(0)-\frac{\d}{\d t}u_{s}^{n_k}(0) \right |>0.	
\end{align}
Then we can modify the initial conditions such that $\bar{\xi}^n(0) \to u(0)$ for some $u(0) \in \Delta^{\S},$ while there is some $t>0$ such that
\begin{align*}
& \limsup_{n \to \infty} \left | \mu_{s}^n(t)-u_s(t) \right |  >0, 
\end{align*}
hence HMFA is inaccurate.	
\end{lemma}

\begin{proof}(Lemma \ref{l:techical})

We can choose $K$ to be large enough such that
\begin{align*}
& \left| \frac{\d^2}{\d t^2}u^{n_k}_s(t) \right |=\left | \frac{\d}{\d t}f_s(u^{n_k}(t)) \right | \leq K
\end{align*}	
also holds uniformly in time and  for any initial conditions. Here we used that $f_s$ is a polynomial and the solutions stay at the simplex $\Delta^{\S}$.

Let $0<t \leq t_0$. Then there is mean value $\tau \in [0,t]$ such that
\begin{align*}
& \mu_s^{n_k}(t)=\mu_{s}^{n_k}(0)+\frac{ \d}{\d t}\mu_{s}^{n_k}(0)t+\frac{1}{2}\frac{\d^2}{\d t^2} \mu_{s}^{n_k}(\tau)t^2\\
& \left|\mu_s^{n_k}(t)-\mu_{s}^{n_k}(0)-\frac{ \d}{\d t}\mu_{s}^{n_k}(0)t \right| \leq\frac{1}{2} \left | \frac{\d^2}{\d t^2} \mu_{s}^{n_k}(\tau)\right | t^2 \leq \frac{K}{2}t^2.
\end{align*}
Similarly,
\begin{align*}
& \left| u_s^{n_k}(t)-u_{s}^{n_k}(0)-\frac{\d}{\d t}u_{s}^{n_k}(0)t \right | \leq \frac{K}{2}t^2.
\end{align*}

Choose a sub-sequence $\left(n_k'\right)_{k=1}^{\infty}$ such that 
\begin{align*}
& \lim_{k \to \infty} \left | \frac{\d}{\d t}\mu_{s}^{n_k'}(0)-\frac{\d}{\d t}u_{s}^{n_k'}(0) \right |= \limsup_{k \to \infty} \left | \frac{\d}{\d t}\mu_{s}^{n_k}(0)-\frac{\d}{\d t}u_{s}^{n_k}(0) \right |=:\varepsilon>0.
\end{align*}
Recall $u_{s}^{n}(0)=\mu_{s}^{n}(0).$ Then
\begin{align*}
\left| \mu_{s}^{n_k'}(t)-u_{s}^{n_k'}(t) \right | \geq  \left | \frac{\d}{\d t}\mu_{s}^{n_k'}(0)-\frac{\d}{\d t}u_{s}^{n_k'}(0) \right |t-Kt^2 \to \varepsilon t-Kt^2.
\end{align*}
Choosing $t = \min \left\{ \frac{\varepsilon}{2K}, t_0\right\}$ yields $\limsup_{k \to \infty} \left| \mu_{s}^{n_k'}(t)-u_{s}^{n_k'}(t) \right |>0.$
From the construction, it is clear that the inequality is also true for any further sub-sequence.

As $\Delta^{S}$ is compact we can choose a further sub-sequence $\left(n_{k}'' \right)$ such that $\mu^{n_k''}(0) \to u(0)$ for some $u(0) \in \Delta^{\S}. $ When $n \neq n_k''$ for some $k$, we modify the initial conditions such that $\mu^{n}(0) \to u(0)$ hold for the whole sequence. Due to the continous dependence on the initial conditions, this also ensures $u^n(t) \to u(t).$ Thus,
\begin{align*}
& \limsup_{n \to \infty} \left | \mu_s^n(t)-u(t) \right|=\limsup_{n \to \infty} \left | \mu_s^n(t)-u^n_s(t) \right| \geq \limsup_{k \to \infty}\left | \mu_s^{n_k''}(t)-u^{n_k''}_s(t) \right|>0.
\end{align*}
\end{proof}

\begin{lemma}
\label{l:del*}

When $\limsup_{n \to \infty} \partial^n_*>0$ the HMFA is not accurate for the degree process.
\end{lemma}

\begin{proof}(Lemma \ref{l:del*})
	
Let $\left(n_{k} \right)_{k=1}^{\infty}$ such that
$\lim_{k \to \infty}\partial^{n_k}_*=\limsup_{n \to \infty}\partial^n_*=:\varepsilon>0.$ 

Recall \eqref{eq:V+}. Choose the initial conditions to be $\xi_{i,a}^n(0)=\1{i \in V_{-}^{n_k}}$ and the rest of to vertices to be $b$. Based on \eqref{eq:del*_nice}, $\left| \delta(V_{-}^{n_k},[N]) \right|=\partial_{*}^{n_k} \to \varepsilon.$

The expectation and the solution for the ODE can be calculated explicitly.
\begin{align*}
 \E \left( \xi_{a}^{n_k}(t) \right)=&\xi_{a}^{n_k}(0)e^{-\frac{d^{n_k}(i)}{\bar{d}^{n_k}}t} \\
 \mu_{a}^{n_k}(t)=&\frac{1}{N(n_k)}\sum_{i \in V_{-}^{n_k}} e^{-\frac{d^{n_k}(i)}{\bar{d}^{n_k}}t} \\
 u_{a}^{n_k}(t)=&\mu_{a}^{n_k}(0)e^{-t}
\end{align*}

 $d^{n_k}(i)<\bar{d}^{n_k}$ yields \eqref{eq:tech1} as
\begin{align*}
	\left | \frac{\d^2}{\d t ^2} \mu_{a}^{n_k}(t) \right |=&\frac{1}{N(n_k)}\sum_{i \in V_{-}^{n_k}} \left(\frac{d^{n_k}(i)}{\bar{d}^{n_k}}\right)^2 e^{-\frac{d^{n_k}(i)}{\bar{d}^{n_k}}t}  \leq \\
	& \frac{1}{N(n_k)}\sum_{i \in V_{-}^{n_k}} 1= \mu_{a}^{n_k}(0) \leq 1.
\end{align*}
As for \eqref{eq:tech2},
\begin{align*}
& \frac{\d}{\d t}\mu_{a}^{n_k}(0)=-\frac{1}{N(n_k)}\sum_{i \in V_{-}^{n_k}} \frac{d^{n_k}(i)}{\bar{d}^{n_k}}=-\frac{e\left(V_{-}^{n_k}, [N]\right)}{N(n_k)\bar{d}^{n_k}} \\
& \frac{\d}{\d t}u_{a}^{n_k}(0)=-u_{a}^{n_k}(0)=-\mu_{a}^{n_k}(0)=-\frac{\left| V_{-}^{n_k} \right|}{N(n_k)} \\
& \left |\frac{\d}{\d t}\mu_{a}^{n_k}(0)-\frac{\d}{\d t}u_{a}^{n_k}(0) \right |=\left|\delta \left(V_{-}^{n_k},[N]\right) \right| \to \varepsilon>0.
\end{align*}
\end{proof}

For technical reasons, we assume $\liminf_{n \to \infty} \bar{d}^n \geq 1$. Luckily, this is not a too restrictive condition based on the lemma below.
\begin{lemma}
\label{l:degree}

Assume $\liminf_{n \to \infty} \bar{d}^n <1.$ Then the HMFA is not accurate for the SI process.
\end{lemma}

\begin{proof}(Lemma \ref{l:degree})

Let $\left(n_k\right)_{k=1}^{\infty}$ a sub-sequence with the property $\lim_{k \to \infty}\bar{d}^{n_k}=\liminf_{n \to \infty} \bar{d}^{n}=:d_0<1.$
\begin{align*}
&\bar{d}^{n_k}= \frac{1}{N(n_k)}\sum_{i=1}^{N(n_k)}d^{n_k}(i)\geq \frac{1}{N(n_k)}\sum_{d^{n_k}(i)\geq 1}d^{n_k}(i) \geq \frac{1}{N(n_k)}\sum_{d^{n_k}(i)\geq 1} 1 \\
& \frac{1}{N(n_k)}\left | \left \{ i \in [N(n_k)] \left| d^{n_k}(i)=0 \right. \right\} \right | \geq 1-\bar{d^{n_k}} \to 1-d_{0}>0,
\end{align*}
Therefore, there exists $A^{n_k} \subset [N(n_k)]$ such that $d^{n_k}(i)=0$ for any $i \in A^{n_k}$ and $\frac{\left| A^{n_k} \right|}{N(n_k)} \to 1-d_0.$

Let examine the SI process whit initial conditions $\xi_{i,I}^{n_k}(0)=\1{i \in A^{n_k}}$ and the rest of the vertices are susceptible. As all the infected vertices are isolated the dynamics in the Markov-process remains constant $\mu_{I}^{n_k}(t)\equiv \mu_{I}^{n_k}(0)$. For the ODE on the other hand 
$$\frac{\d}{\d t}u_{I}^{n_k}(0)=\beta \mu_{I}^{n_k}(0) \left(1-\mu_{I}^{n_k}(0)\right) \to \beta d_0(1-d_0)>0,$$
so the conditions of Lemma \ref{l:techical} clearly hold.
\end{proof}

\begin{lemma}(Getting rid of the extreme vertices)
\label{l:thining}

Assume $\limsup_{n \to \infty} \partial_2^n>0$ and $\partial_*^n \to 0$. $A^n, B^n \subset [N]$ are disjoint sets such that $\left| \delta(A^n,B^n) \right|=\partial_2^n$.
\begin{align*}
& D^{n}:= \left \{ i \in [N] \left |d^n(i) \leq 2 \bar{d}^{n} \right. \right\} \\
& A_{D}^n:=A^n \cap D^n \\
& B_{D}^n:=B^n \cap D^n
\end{align*}
Then $\limsup_{n \to \infty} \left| \delta\left(A_{D}^n,B_{D}^n\right) \right |>0.$
\end{lemma}

\begin{proof}(Lemma \ref{l:thining})
	
First, we show that $\limsup_{n \to \infty} \left| \delta \left(A_{D}^{n},B^n\right) \right |>0$. Introduce
\begin{align*}
& \bar{D}^{n}:= \left \{ i \in [N] \left |d^n(i) > 2 \bar{d}^n \right. \right\} \\
& \bar{A}_{D}^n:=A^n \cap \bar{D}^n	
\end{align*}
Clearly $A^n_{D} \sqcup \bar{A}^n_{D}=A^n$ and $\bar{A}_{D}^n \subset \bar{D}^n.$

Let $\iota^n$ be a uniform random variable on $[N]$.
\begin{align*}
 \frac{\left |\bar{D}^n \right |}{N}=&\pr\left(d^n(\iota^n)>2 \bar{d}^n \right) \leq \pr \left( \left |d^n(\iota^n)-\bar{d}^n \right | > \bar{d}^n\right)	 \leq \\
 & \frac{1}{\bar{d}^n}\E \left | d^n(\iota^n)-\bar{d}^n \right |=\frac{1}{N \bar{d}^n}\sum_{i=1}^{n} \left |d^n(i)- \bar{d}^n\right |=2 \partial_*^n \to 0
\end{align*}

Assume indirectly $\left| \delta \left(A_{D}^{n},B^n\right) \right | \to 0.$ Therefore $\limsup_{n \to \infty} \left |\delta(A^n,B^n) \right |>0$ imply $\limsup_{n \to \infty} \left |\delta\left(\bar{A}^n_D,B^n\right) \right |>0.$ On the other hand, 
\begin{align*}
& \limsup_{n \to \infty} \left |\delta\left(\bar{A}^n_{D},B^n\right) \right |=\limsup_{n \to \infty} \left | \frac{e\left(\bar{A}_{D}^n, B^n\right)}{N \bar{d}^n}-\frac{\left|\bar{A}_{D}^n \right|}{N} \frac{\left | B^n \right |}{N} \right |= \\
&\limsup_{n \to \infty}  \frac{e\left(\bar{A}_{D}^n, B^n\right)}{N \bar{d}^n} \leq \limsup_{n \to \infty}  \frac{e\left(\bar{D}^n, [N]\right)}{N \bar{d}^n}= \\
& \limsup_{n \to \infty} \left | \frac{e\left(\bar{D}^n, [N]\right)}{N \bar{d}^n}-\frac{\left|\bar{D}^n \right|}{N} \right |=\limsup_{n \to \infty} \left | \delta \left( \bar{D}^n, [N]\right) \right | \leq \limsup_{n \to \infty} \partial_{*}^n=0,
\end{align*}
resulting in a contradiction.

$\limsup_{n \to \infty} \left | \delta \left(A_{D}^n, B_{D}^n \right) \right |>0$ follows from the same argument in the second variable.
\end{proof}

Whit this tool is hand we can finally prove Theorem \ref{t:negative}.

\begin{proof} (Theorem \ref{t:negative} )
	
$\limsup_{n \to  \infty } \partial^n>0$ is assumed. We may also assume $\partial^n_* \to 0$ and $\limsup_{n \to \infty} \bar{d}^n \geq 1,$ otherwise, Lemma \ref{l:del*} and \ref{l:degree} proves the statement. Based on $\partial_*^n \to 0$ and Lemma \ref{l:hierarchy} $\partial_{2}^n, \partial^n$ are equivalent, hence $\limsup_{n \to \infty } \partial^n_2>0.$

Based on Lemma \ref{l:thining} there are disjoint sets $A_{D}^n,B_{D}^n \subset [N]$ such that $d^n(i) \leq 2 \bar{d}^n$ for all $i \in A_{D}^n,B_{D}^n $ and
$$\delta_0:=\limsup_{n \to \infty} \left | \delta(A^n_{D},B_{D}^n) \right |>0. $$
We can choose a sub-sequence $\left(n_k\right)_{k=1}^{\infty}$ such that  $\left | \delta(A^{n_k}_{D},B_{D}^{n_k}) \right | \to \delta_0$
and $\bar{d}^{n_k} \geq \frac{1}{2}$ as we have $\limsup_{n \to \infty} \bar{d}^n \geq 1$ as well.

Define an SI process whit initial conditions $V_S^n(0)=A_D^n, \ V_{I}^n(0)=B_D^n$ and the rest of the vertices are in state $R$. Our goal is to show that the conditions in Lemma \ref{l:techical} are satisfied for state $S$.

The derivatives are
\begin{align*}
&\frac{\d}{\d}\mu_{S}^{n}(t)=-\beta \E \left( \nu_{SI}^n(t)\right),\\
& \frac{\d}{\d t}u^n_S(t)=-\beta u_{S}^n(t)u_{I}^n(t).
\end{align*}
For time $t=0$
\begin{align*}
& \left| \frac{\d}{\d}\mu_{S}^{n_k}(0)-\frac{\d }{\d t}u_{S}^{n_k}(0) \right|=\beta \left | \delta \left(A_{D}^{n_k},B^{n_k}_D\right) \right | \to \beta \delta_0>0,
\end{align*}
hence \eqref{eq:tech2} is satisfied. For \eqref{eq:tech1} we  use \eqref{eq:simon}.
\begin{align*}
 &\left | \frac{\d^2}{\d t^2}\mu_{S}^{n_k}(t) \right|= \beta \left| \frac{\d}{ \d t} \nu_{SI}^{n_k}(t) \right |=\frac{\beta}{N \bar{d}^{n_k}} \left | \frac{\d}{\d t}[SI]^{n_k}(t) \right| \leq \\
 & \frac{\beta^2}{N \left(\bar{d}^{n_k} \right)^2} \left([SSI]^{n_k}(t)+[ISI]^{n_k}(t)+[SI]^{n_k}(t) \right)
\end{align*}

Recall the definition of $D^n$ in Lemma \ref{l:thining}. Since the set of vertices in state $S$ can only decrease $V_{S}^n(t) \subset V_{S}^n(0) \subset D^n$. As for state $I$, it can only increase, however no other extra vertices can become infected then those who where initially susceptible, hence $V_{I}^n(t) \subset V_{S}^n(0) \sqcup V_{I}^n(0) \subset D^n.$

\begin{align*}
&\frac{1}{N \left(\bar{d}^{n_k} \right)^2}[SSI]^{n_k}(t)=\frac{1}{N \left(\bar{d}^{n_k} \right)^2} \E \left( \sum_{i=1}^{N}\sum_{j=1}^{N}\sum_{l=1}^{N}a_{ij}^{n_k} a_{jl}^{n_k} \xi_{i,S}^{n_k}(t)\xi_{j,S}^{n_k}(t)\xi_{l,I}^{n_k}(t) \right) \leq \\
& \frac{1}{N \left(\bar{d}^{n_k} \right)^2}  \sum_{i=1}^{N}\sum_{j \in D^n}\sum_{l=1}^{N}a_{ij}^{n_k} a_{jl}^{n_k}  =\frac{1}{N \left(\bar{d}^{n_k} \right)^2}\sum_{j \in D^n} \left(d^{n_k}(j) \right)^2 \leq \frac{1}{N}\sum_{j \in D^n} 4 \leq 4.
\end{align*}
The same bound goes for $\frac{1}{N \left(\bar{d}^{n_k} \right)^2}[ISI]^{n_k}(t).$ As for the last term,
\begin{align*}
& \frac{1}{N \left(\bar{d}^{n_k} \right)^2}[SI]^{n_k}(t)= \frac{\nu_{SI}^{n_k}(t)}{\bar{d}^{n_k}} \leq 2,	
\end{align*}
where we used $\bar{d}^{n_k} \geq \frac{1}{2}$.		
\end{proof}

Lastly, we prove Theorem \ref{t:not_quasi_random}.

\begin{proof}(Theorem \ref{t:not_quasi_random})
	
	Let $\left(n_k \right)_{k=1}^{\infty}$ be a sub-sequence for which $G^{n_k}$ is bipartite. Then $[N(n_k)]$ can be split into $[N(n_k)]=V_1^{n_k} \sqcup V_2^{n_k}$ such that all of the edges go between  $V_1^{n_k}$ and $V_2^{n_k}$, thus $e\left(V_1^{n_k},V_2^{n_k} \right)=N \bar{d}^{n_k}.$
	\begin{align*}
		& \partial^{n_k} \geq \left|\delta\left(V_1^{n_k},V_2^{n_k}\right)\right|=1-\frac{\left|V_1^{n_k}\right|}{N} \left(1-\frac{\left|V_1^{n_k}\right|}{N}\right) \geq 1-\left( \frac{1}{2} \right)^2=\frac{3}{4}.
	\end{align*}
	Thus 1) is true.
	
	Define $\theta_{sup}:=\limsup_{n \to \infty}\theta^n>0$ and $\left(n_k\right)_{k=1}^{\infty}$ a sub-sequence such that $\theta^{n_k} \to \theta_{sup}.$ 
	
	At first, assume $0<\theta_{sup}<1.$ 
	\begin{align*}
		\partial^{n_k} \geq & \left | \delta\left(V_{\textrm{conn}}^{n_k},[N] \setminus V_{\textrm{conn}}^{n_k} \right) \right |=\frac{\left|V_{\textrm{conn}}^{n_k} \right|}{N(n_k)} \left(1-\frac{\left|V_{\textrm{conn}}^{n_k} \right|}{N(n_k)}\right)=\\
		&\theta^{n_k}\left(1-\theta^{n_k} \right ) \to \theta_{sup} \left(1-\theta_{sup}\right)>0
	\end{align*}
	
	When $\theta_{sup}=1$, then even the largest connected component covers only $o\left(N(n_k)\right)$ vertices. Let the connected components be $V_1^{n_k} , \dots, V_{Q_{n_k}}^{n_k}$ ordered increasingly. All of them contains at most $o\left(N(n_k)\right)$ vertices. 
	\begin{align*}
		& j^{n_k}:\min \left \{ j \geq 1 \left | \sum_{l=1}^{j} \left|V_l^{n_k} \right| \geq \frac{N(n_k)}{2} \right.  \right\} \\
		& A^{n_k}:= \bigsqcup_{i=1}^{j^{n_k}}V_l^{n_k}
	\end{align*}  
	Such a set must contain $\left |A^{n_k} \right |=\frac{N(n_k)}{2}+o\left(N(n_k)\right)$ vertices and does not share edges whit its complement. Using this set and its complement concludes the proof of 2) as
	\begin{align*}
		\partial^{n_k} \geq \left | \delta\left(A^{n_k},[N] \setminus A^{n_k}  \right) \right |=\frac{\left| A^{n_k} \right |}{N(n_k)} \left(1-\frac{\left| A^{n_k} \right |}{N(n_k)} \right)=\frac{1}{4}+o(1) \to \frac{1}{4}.
	\end{align*}
	
	For 3), let $(n_k)_{k=1}^{\infty}$ be a sequence such that $ \lim_{k \to \infty}\frac{\alpha^{n_k}}{N(n_k)}=\limsup_{n \to \infty} \frac{\alpha^{n}}{N}=:\alpha^*>0. $ $A^{n_k}$ is a sequence of independent set with size $\left|A^{n_k} \right|=\alpha^{n_k}.$ 
	\begin{align*}
		\partial^{n_k} \geq \left|\delta\left(A^{n_k},A^{n_k}\right) \right|=\left(\frac{\left|A^{n_k} \right|}{N(n_k)}\right)^2 \to \left(\alpha^*\right)^2>0.
	\end{align*}
	
	Lastly, $\left(n_k\right)_{k=1}^{\infty}$ is a sub-sequence with $\lim_{k \to \infty}\bar{d}^{n_k}=\liminf_{n \to \infty}\left(D^{n}\right)^c=:D.$ Then according to \eqref{alpha_bound}
	$$ \frac{\alpha^{n_k}}{N(n_k)} \geq \frac{1}{\bar{d}^{n_k}+1} \to \frac{1}{D+1}>0, $$
	hence the condition for 3) holds proving 4).
\end{proof}

\section{Conclusion}

In this work the accuracy of the homogeneous mean field approximation of density-dependent Markov population processes whit linear transition rates was studied	on graph sequences. The motivation for examining HMFA was giving explicit error bounds for graphs sequences which are only approximately homogeneous, like expanders, and to gain insight for other, more complex  approximations like  IMFA  for which HMFA is a special case under certain conditions.

A characterization was given, namely the large graph limit is the ODE given by HMFA for any such Markov process and convergent initial condition if an only if the graph sequence is  quasi-random (has vanishing discrepancy). 

Explicit error bound whit order $\frac{1}{\sqrt{N}}$ plus the discrepancy was shown. The discrepancy can be further upper bounded by the spectral gap plus a quantity measuring the heterogeneity of the degree distribution, meaning, when the graph is "random enough" and degrees are homogeneous then HMFA performs well.

For Erdős-Rényi and for random regular graph-sequences we show the upper bound is of order the inverse square root of the average degree. Therefore, diverging average degrees are sufficient for vanishing error for these cases. It has also been shown that diverging average degrees are necessary for quasi-randomness.

In future works, we hope to extend our results to more complex approximations like IMFA.

\subsection*{Acknowledgment}
The author is thankful for Illés Horváth, who introduced the topic and helped whit editing the article.

\bibliographystyle{abbrv}
\bibliography{salad}

\end{document}